\documentclass[10pt,reqno]{amsart}
\usepackage{a4wide}
\allowdisplaybreaks[1]
\usepackage{amssymb}
\usepackage[all]{xy}
\DeclareMathOperator{\C}{C}
\DeclareMathOperator{\R}{R}

\DeclareMathOperator{\Sp}{Sp}

\DeclareMathOperator{\G}{G}
\DeclareMathOperator{\M}{M} \DeclareMathOperator{\Mor}{Mor}
\DeclareMathOperator{\B}{B} \DeclareMathOperator{\Rep}{Rep}
 \DeclareMathOperator{\Aut}{Aut}
\DeclareMathOperator{\W}{L}
\DeclareMathOperator{\X}{X}

\DeclareMathOperator{\Ind}{Ind}
\DeclareMathOperator*{\tens}{\otimes}

\newtheorem{thm}{Theorem}[section]
\newtheorem{lem}[thm]{Lemma}
\newtheorem{prp}[thm]{Proposition}
\theoremstyle{definition}
\newtheorem{dfn}[thm]{Definition}

\newtheorem{rem}[thm]{Remark}
\newtheorem{ntn}[thm]{Notation}
\newcommand{\cml}{\mathcal{K}(\W^2(\mbG))}

\newcommand{\ev}{{\rm ev}}

\newcommand{\mbG}{\mathbb{G}}
\newcommand{\mbX}{\mathbb{X}}
\newcommand{\hg}{\hat\gamma}
\newcommand{\HG}{\Hat\Gamma}
\newcommand{\mbD}{\mathbb{D}}
\newcommand{\mbB}{\mathbb{B}}
\newcommand{\tp}{\xymatrix{*+<.7ex>[o][F-]{\scriptstyle\top}}}

\begin{document}
\title[]{On a certain approach to quantum homogeneous spaces}
\author{P.~Kasprzak}
\address{Department of Mathematical Methods in Physics, Faculty of Physics, Warsaw University, Warsaw, Poland}
\thanks{Supported by the Marie Curie Research Training Network Non-Commutative Geometry MRTN-CT-2006-031962}
\thanks{Supported by Geometry and Symmetry of Quantum Spaces,
PIRSES-GA-2008-230836}
\email{pawel.kasprzak@fuw.edu.pl}
\begin{abstract} 
We propose a definition of a quantum homogeneous space of a locally compact quantum group. We show that classically it reduces to the notion of a homogeneous spaces. On the quantum level our definition goes beyond the quotient case. It provides a framework  which, besides the Vaes' quotient of a locally compact quantum group by its closed quantum subgroup (our main motivation) is also compatible  with, generically non-quotient, quantum homogeneous spaces of a compact
quantum group  studied
by P. Podle\'s  as well as the Rieffel deformation of
G-homogeneous
spaces. Finally, our definition rules out the paradoxical examples of the non-compact quantum homogeneous spaces of a compact quantum group. 
\end{abstract}
 \maketitle
\begin{section}{Introduction}
Although the theory of locally compact quantum groups is already well established (see \cite{KV}, \cite{MNW}), the quantum counterpart of the quantum homogeneous space is not yet known. One aspect of quantum homogeneous spaces  which is not thoroughly understood is related with the Vaes' construction  \cite{Vaes}  of the quotient of a locally compact quantum group group by its closed quantum subgroup: in order to prove the existence of the quantum quotient, the regularity of a given quantum group is needed. The second difficulty  is related with the non-quotient type of a generic quantum homogeneous space which is  expected since the construction of Podle\'s spheres in \cite{Pod}. 

In the case of compact quantum groups the situation considerably simplifies. It easy to realize that the quantum counterpart of a transitive group action  is provided by the concept of ergodic coaction (see  \cite{BF}, \cite{Pod}). Adopting this view point, the harmonic analysis on the homogeneous spaces may be extended to the quantum  setting. A quantum counterpart  of a given classical result may happen to be not totally  straightforward which is exemplified  by the following surprising fact:  the multiplicity of an irreducible subrepresentation entering a quantum homogeneous space, though always finite, may exceed the classical dimension of this representation - for the explicit examples we refer to \cite{VaesErg}. In order to formulate the quantum version of the discussed inequality the quantum dimension was introduced in \cite{BF}. In spite of some subtleties, the theory of quantum homogeneous spaces of compact quantum groups is  satisfactory. 

In this paper we propose a definition of a quantum homogeneous space (QHS) which goes beyond the compact and the quotient case. Our approach is motivated by the Vaes' construction of a quantum quotient of a locally compact quantum group by its closed quantum subgroup and is based on the interplay between the $\C^*$-algebraic and the von Neumann version of a QHS. Classically, our definition  reduces to the notion of homogeneous spaces.  Our definition is compatible with different classes of examples of quantum homogeneous spaces such as: quantum homogeneous spaces of the quotient type due to S. Vaes \cite{Vaes}, quantum homogeneous spaces of a compact quantum group  studied by P. Podle\'s \cite{Pod1} and the Rieffel deformation of
G-homogeneous spaces - see \cite{pkrdhs}. On the other hand we are able to rule out the paradoxical examples of non-compact quantum homogeneous spaces of a compact quantum group provided by S.L. Woronowicz \cite{woreq2}. 

Although the above advantages, the proposed theory of QHS is  not yet satisfactory. It is expected that in the non-regular quotient case there should exist  examples which do not fit into it, not mentioning that dropping  the regularity assumption one does not even know how to construct the quantum quotient. Some of the experts suggested  to base the construction of the quotient-type QHS on the integration along the quantum subgroup (private communication by S.L. Woronowicz, R. Meyer). The mathematical tools for this construction already exist (see \cite{RM}, \cite{Rf3}) but the theory of QHS which uses it is not formulated. 

Let us describe the contents of the paper. In Section \ref{prem} we recall the basic definitions and  fix a notation concerning locally compact quantum groups. In particular we specify the category $\mathcal{C}(\mbG)$ of $\mbG$-$\C^*$-algebras. In Section \ref{secqhs} the definition of a quantum homogeneous space is formulated. As was already mentioned our approach is based on the interplay between the $\C^*$- and the von Neumann algebraic version of a given QHS. We prove that they  uniquely determine each other. In Section \ref{minsec} we show that the $\C^*$-algebraic version of a QHS is a $\mbG$-simple object in $\mathcal{C}(\mbG)$. This implication is a quantum counterpart of the classically trivial fact that the homogeneity of a space is a stronger property than the density of each orbit. It may be surprising  that in order to prove the quantum version of this implication one has to use the Tomita-Takesaki theory employed in the construction of the canonical implementation of a coaction (see \cite{Vaesci}). In Section \ref{seccl} we prove that there is a 1-1 correspondence between the  homogeneous $\G$-spaces   and quantum homogeneous $\mbG$-spaces  with the underlying von Neumann algebra being commutative. In Section \ref{secqhsp} we show that our notion of a QHS when restricted to compact quantum groups boils down to the standard notion of a quantum homogeneous space introduced by P. Podle\'s \cite{Pod1}.  Section \ref{secrd} contains the proof of the fact that the Rieffel deformation of a homogeneous space is a quantum homogeneous space. It is a generalization of the result of \cite{pkrdhs} which deals with the case when the deformed QHS is of the quotient type.  

Throughout the paper we freely use the $\C^*$-algebraic concepts  such as multipliers $\M(A)$ of a given $\C^*$-algebra $A$ and morphisms $\Mor(A,B)$ from $A$ to a $\C^*$-algebra $B$. Our main reference for the $\C^*$-algebraic notions is \cite{W4}. For the theory of von Neumann algebras we refer to one of the standard textbooks (e.g.  \cite{sun}). We often use the {\it leg numbering notation}. Consider a tensor triple of Hilbert spaces $\mathcal{H}_1\otimes\mathcal{H}_2\otimes\mathcal{H}_3$. For $X\in\B(\mathcal{H}_1\otimes\mathcal{H}_2)$ we define $X_{12}=X\otimes 1$. We analogously define $Y_{23}$ and $Z_{13}$. The leg numbering notation passes naturally to the context of elements of the triple tensor products of $\C^*$- or von Neumann algebras. 

I would like to express my gratitude to A. Skalski for reading a preliminary version of this paper and giving the hints that helped to improve it. 
\end{section}
\begin{section}{Preliminaries}\label{prem}
\begin{dfn}
Let $\mathrm{G}$ be a locally compact group and $\X$ be a locally compact space. A left action of $\mathrm{G}$ on $\X$ is a continuous map $\mathrm{G}\times \X\rightarrow \X$ denoted by $(g,x)\mapsto gx$ satisfying
\begin{itemize}
\item[(i)] for all $g\in \mathrm{G}$, the map $\X\ni x\mapsto gx\in \X$ is a homeomorphism of $\X$, 
\item[(ii)] for all $g_1,g_2\in \mathrm{G}$ and $x\in \X$, $(g_1g_2)x = g_1(g_2x)$.
\end{itemize} A locally compact space $\X$ equipped with an action of $\mathrm{G}$ is called a $\mathrm{G}$-space.
A $\mathrm{G}$-space is said to be transitive   if for each pair $x_1,x_2\in \X$ there exists $g\in \mathrm{G}$ such that $gx_1=x_2$. 
\end{dfn}

Let $\mathrm{G}$ be a locally compact group and let $\mathrm{H}$ be a closed subgroup of $\mathrm{G}$. The quotient $\mathrm{G}/\mathrm{H}$ with the standard left action of $\mathrm{G}$ is an example of a transitive $\mathrm{G}$-space. 

Let $\X$ be a transitive  $\G$-space. A point $x\in \X$ is assigned with its isotropy group: $\mathrm{H}_x=\{g\in \G:gx=x\}$. The group $\mathrm{H}_x\subset \G$ is closed and we have a bijection $\Phi_x: \mathrm{G}/\mathrm{H}_x\in g\mathrm{H}_x\mapsto gx\in \X$. This  map is a continuous bijection which intertwines the actions of $\mathrm{G}$. It is easy to see that  if $\Phi_x$ is a homeomorphism for one point $x_0\in \X$ then  $\Phi_x$ is a homeomorphism for all $x\in \X$. 
\begin{dfn} We say that a transitive $\mathrm{G}$-space $\X$  is homogeneous (or $\X$ is a homogeneous $\mathrm{G}$-space) if the map $\Phi_x$ is homeomorphism.
\end{dfn}
As was shown in \cite{Lindblad} if $\mathrm{G}$ is $\sigma$-compact then all transitive $\mathrm{G}$-spaces are homogeneous. In order to see an example of a transitive  $\mathrm{G}$-space  which is not homogeneous let us consider the group $\mathbb{R}_d$ with the discrete topology. The real line $\mathbb{R}$ with the action 
\[\mathbb{R}_d\times\mathbb{R}\ni(x_1,x_2)\mapsto x_1+x_2\in\mathbb{R}\] is a transitive $\mathbb{R}_d$-space. Taking $x = 0$, we get  $\Phi_0:\mathbb{R}_d\rightarrow\mathbb{R}$,   
$\Phi_0(x) = x$ which is not homeomorphism. Note that $\mathbb{R}_d$ in this example is not $\sigma$-compact. 

In order to define quantum homogeneous spaces we shall first recall basic definitions and facts  concerning locally compact quantum groups (LCQG). Our main reference for this subject is \cite{KV} and \cite{kusvaes}. This section to some extents follows Section 2 of \cite{Vaes}. Let us emphasize that our definition of a quantum homogeneous space allows a non-quotient type examples.

Since a definition of LCQGs is based on the existence of Haar weights let us first recall some weight theoretic concepts. Let $\phi$ be a normal, semi-finite, faithful (n.s.f.) weight on a von Neumann algebra $M$. Then we shall denote:
\[\mathcal{M}_\phi^+=\{x\in M^+|\,\phi(x)<\infty\}\mbox{ and }\mathcal{N}_\phi=\{x\in M|\,\phi(x^*x)<\infty\}.\]
\begin{dfn}\label{lcqg}
A pair $\mbG=(M,\Delta)$ is called a locally compact quantum group  when
\begin{itemize}
\item $M$ is a von Neumann algebra and $\Delta:M\rightarrow M\bar\tens M$ is a normal and unital $*$-homomorphism satisfying the coassociativity relation: $(\Delta\otimes\iota)\circ\Delta=(\iota\otimes\Delta)\circ\Delta$;
\item there exist n.s.f. weights $\phi$ and $\psi$ on $M$ such that 
\begin{itemize}
\item[-] $\phi$ is left invariant: $\phi((\omega\otimes\iota)\Delta(x))=\phi(x)\omega(1)$ for all $x\in\mathcal{M}_\phi^+$ and $\omega\in M_*^+$,
\item[-] $\psi$ is right invariant: $\psi((\iota\otimes\omega)\Delta(x))=\psi(x)\omega(1)$ for all $x\in\mathcal{M}_\psi^+$ and $\omega\in M_*^+$.
\end{itemize}
\end{itemize}
\end{dfn}
In the first step of the development of the theory one moves on to the  Hilbert space level. Let  ($\W^2(\mbG)$,\,$\pi$,\,$\Lambda_\phi$) be the  GNS-triple corresponding to the weight $\phi$, where $\W^2(\mbG)$ is the completion of $\mathcal{N}_{\phi}$ in the norm $||x||^2_{\phi}=\phi(x^*x)$,  $\pi:M\rightarrow\B(\W^2(\mbG))$ is the GNS-representation and $\Lambda_\phi:\mathcal{N}_\phi\rightarrow\W^2(\mbG)$ is the GNS-map. The  multiplicative unitary operator   $W\in\B(\W^2(\mbG)\otimes\W^2(\mbG))$, also referred as the left regular corepresentation, is defined by the formula: \[W^*(\Lambda_\phi(a)\otimes\Lambda_\phi(b))=(\Lambda_\phi\otimes\Lambda_\phi)(\Delta(b)(a\otimes 1))\mbox{ for any } a,b\in\mathcal{N}_\phi.\] One proves that the von Neumann algebra $M$  may be recovered as the  strong closure of the algebra $\{(\iota\otimes\omega)(W)\,|\,\omega\in\B(\W^2(\mbG))_*\}$. In turn, the comultiplication $\Delta$ is implemented by $W$   \[\Delta(x)=W^*(1\otimes x)W \mbox{ for any } x\in M.\] The operator $W$ encodes the pair $(M,\Delta)$. 

One of the trademarks of the  LCQG theory is a well established duality  - a broad generalization of the Pontryagin duality for locally compact abelian groups. Starting with $\mbG$ one  constructs the dual quantum group $\widehat\mbG=(\Hat{M},\Hat\Delta)$. The duality, on the level of the multiplicative unitaries  is manifested as follows: the multiplicative unitary related with $\widehat\mbG$ is given by $\hat{W}=\Sigma W^* \Sigma$ where $\Sigma:\W^2(\mbG)\otimes\W^2(\mbG)\rightarrow \W^2(\mbG)\otimes\W^2(\mbG)$ is the flip operator. The von Neumann algebra $\Hat{M}$ is the strong closure of the algebra $\{(\iota\otimes\omega)(\Hat W)\,|\,\omega\in\B(\W^2(\mbG))_*\}$ and the comultiplication $\Hat{\Delta}$ is implemented by $\Hat{W}$: $\Hat\Delta(x)=\Hat{W}^*(1\otimes x)\Hat{W}$ for all $x\in\Hat{M}$. It is a nontrivial fact that $\widehat\mbG$ is a LCQG in the sense of Definition \ref{lcqg}, i.e. there exist the canonically defined left and right invariant weights on $\Hat{M}$ denoted by  $\hat\phi$ and $\hat\psi$ respectively. Let us write $J$ and $\hat{J}$ for the modular conjugations of the weights $\phi$ and $\hat\phi$ respectively. Using them one introduces the right regular corepresentations $V\in\Hat{M}'\otimes M$ and $\Hat{V}\in M'\otimes \Hat{M}$:
\begin{equation}\label{rrr}V=(\Hat{J}\otimes\Hat{J})\Hat{W}(\Hat{J}\otimes\Hat{J}),\,\,\,\,\Hat{V}=(J\otimes J)W(J\otimes J).\end{equation}

It turns out that $\mbG=(M,\Delta)$, which is an object of the von Neumann type, encodes also the $\C^*$-algebraic version of $\mbG$. To be more precise,  the $\C^*$-algebraic counterpart $(A,\Delta)$ of $\mbG$   may be recovered from $(M,\Delta)$ by means of the multiplicative unitary $W$. The $\C^*$-algebra $A$ coincides with the norm closure of the algebra
$\{(\iota\otimes\omega)(W)\,|\,\omega\in\B(\W^2(\mbG))_*\}$ and $\Delta$ restricts to a  morphism $\Delta\in\Mor(A, A\otimes A)$. When there is no danger of confusion, we shall also write  $\mbG=(A,\Delta)$.

Let us move on to the coactions of quantum groups on von Neumann algebras.
\begin{dfn}\label{vncoac}
Let $\mbG=(M,\Delta)$ be a locally compact quantum group, $N$ a von Neumann algebra and let $\alpha:N\rightarrow M\bar\tens N$ be a normal, injective,  unital $*$-homomorphism. We say that $\alpha$ is a left coaction of $\mbG$ on $N$  when $(\iota\otimes\alpha)\circ\alpha=(\Delta\otimes\iota)\circ\alpha$. 
We define the algebra of coinvariants $N^{\alpha}$:
\[N^{\alpha}=\{x\in N|\,\alpha(x)=1\otimes x\}.\] We say that $\alpha$ is ergodic when $N^{\alpha}=\mathbb{C}1$. 
\end{dfn}
Let $\alpha:N\rightarrow M\bar\tens N$ be a coaction of $\mbG$,  $\theta:N^+\rightarrow[0,\infty]$  a n.s.f. weight and let $K$ be the Hilbert space of the GNS-construction based on $\theta$. The antilinear Tomita-Takesaki conjugation will be denoted by $J_\theta:K\rightarrow K$. Combining  Proposition 3.7, Proposition 3.12  and Theorem 4.4 of \cite{Vaesci} one may prove the quantum counterpart of the Haagerup's theorem on the existence of the canonical unitary implementation of an action of a locally compact group on a von Neumann algebra.
\begin{thm}\label{unimpth}
Let us adopt the notation introduced above. There exists a unitary operator $U\in M\bar\otimes\B(K)$ such that 
\begin{itemize}
\item $(\Delta\otimes\iota)U=U_{23}U_{13}$;
\item $\alpha(x)=U(1\otimes x)U^*$;
\item $(\hat{J}\otimes J_\theta)U(\hat{J}\otimes J_\theta)=U^*$.
\end{itemize}
\end{thm}
In order to describe the coactions of $\mbG$ on the $\C^*$-algebras we shall introduce the category $\mathcal{C}(\mbG)$ of $\mbG$-$\C^*$-algebras - a quantum counterpart of the category $\mathcal{C}(\G)$ of $\G$-$\C^*$-algebras (see e.g. \cite{KQ}). We adopt the following notation: the closed linear span of a subset $\mathcal{W}\subset\mathcal{V}$ of a Banach space $\mathcal{V}$ will be denoted by $[\mathcal{W}]$. The  definition of $\mbG$-$\C^*$-algebras may be traced back to \cite{BF} and \cite{Pod1}.
\begin{dfn}\label{concoact} Let $\mbG$ be a LCQG. A $\mbG$-$\C^*$-algebra is a  pair $\mbD=(D,\Delta_D)$ consisting of a $\C^*$-algebra $D$ and an injective coaction $\Delta_D\in\Mor(D,A\otimes D)$:
\[(\iota\otimes\Delta_D)\circ \Delta_D = (\Delta\otimes\iota)\circ\Delta_D \] which is continuous: $[\Delta_D(D)(A\otimes 1)]=\C_0(\mbG)\otimes D$. The $\C^*$-algebra $D$ will also be denoted by $\C_0(\mbD)$ and the coaction $\Delta_D$ will be denoted by $\Delta_{\mbD}$.
\end{dfn}
In accordance with the above definition we shall use the notation $\C_0(\mbG)$ and $\Delta_{\mbG}$. We shall also denote $M$ by $\W^\infty(\mbG)$.
In order to specify the morphism in the category of $\mbG$-$\C^*$-algebra we adopt the following definition.
\begin{dfn}\label{Gmor}
Let $\mbG$ be a LCQG and suppose that $\mathbb{B}$ and $\mbD$ are $\mbG$-$\C^*$-algebras. We say that a morphism $\pi\in\Mor(\C_0(\mbB),\C_0(\mbD))$ is a $\mbG$-morphism if $\Delta_{\mbD}\circ\pi=(\iota\otimes\pi)\circ\Delta_{\mbB}$. The set of $\mbG$-morphism from $\mathbb B$ to $\mbD$  will be denoted by $\Mor_{\mbG}(\mathbb{B},\mbD)$.
\end{dfn}

Let $\mbD$ be a $\mbG$-$\C^*$-algebra. One may extend the classical crossed product construction to this case defining the reduced crossed product $\widehat\mbG_{\rm cop}$-$\C^*$-algebra $\mbG\ltimes\C_0(\mbD)$, where $\widehat\mbG_{\rm cop}=(\Hat\M,\Sigma\Hat\Delta(\cdot)\Sigma)$. For the details of this construction we refer to \cite{Vaes}. Let us only recall the definition of the  $\C^*$-algebra $\mbG\ltimes\C_0(\mbD)$:
\begin{equation}\label{cacp}
\mbG\ltimes\C_0(\mbD)=[\Delta_{\mbD}(\C_0(\mbD))(\C_0(\widehat\mbG)\otimes 1)].
\end{equation}

In the course of the paper we shall represent the von Neumann algebras on the $\C^*$-Hilbert modules - the reference for the subject is the standard textbook  \cite{EL}. Let us remind  that for any $\C^*$-Hilbert module $\mathcal{E}$ we may define the $\C^*$-algebra of adjointable operators: \[\mathcal{L}(\mathcal{E})=\{T:\mathcal{E}\rightarrow\mathcal{E}\,|\,\exists\,\, T^*:\mathcal{E}\rightarrow\mathcal{E} \mbox{ s.t. }(e_1|Te_2)=(T^*e_1|e_2) \mbox{ for any }e_1,e_2\in\mathcal{E}\}.\] Let $\mathcal{K}(\mathcal{E})\subset\mathcal{L}(\mathcal{E})$ be the ideal of compact operators. We shall often identify $\mathcal{L}(\mathcal{E})$ with $\M(\mathcal{K}(\mathcal{E}))$. 
The following definition is due to S. Vaes - see Definition 3.1 \cite{Vaes}.
\begin{dfn} Let $N$ be a von Neumann algebra and $\mathcal{E}$ a $\C^*$-Hilbert module. A unital $*$-homomorphism $\pi:N\rightarrow\mathcal{L}(\mathcal{E})$ is said to be strict if it is $*$-strongly continuous on the unit ball of $N$. 
\end{dfn}
The strictness of a $*$-homomorphism $\pi:N\rightarrow\mathcal{L}(\mathcal{E})$ means that for any $*$-strongly convergent, uniformly bounded net  $x_i\in N$ and any $v\in\mathcal{E}$, the nets $\pi(x_i)v$ and $\pi(x_i^*)v$ are norm convergent in $\mathcal{E}$. 
\begin{rem}\label{2id}
Let $N$ be a von Neumann algebra and $p\in N$ a central projection in $N$. We may define  a weakly closed $2$-sided ideal generated by $p$: $I=pN$. This  establishes a 1-1 correspondence between the central projections in $N$ and  its $2$-sided ideals.
\end{rem}
\end{section}
\begin{section}{Quantum homogeneous spaces - definition and the uniqueness results}\label{secqhs} 
Let $\mbG$ be a locally compact quantum group.
Our definition of a quantum homogeneous space is motivated by the Vaes' construction of the quotient of a locally compact quantum group by a closed quantum subgroup and is based  on the interplay between the $\C^*$-algebraic and the von Neumann version of it. Let us first formulate the definition leaving the explanation of  motivation  for later.
\begin{dfn}\label{qhs2} Let $N$ be a von Neumann algebra and let $\Delta_N:N\rightarrow \W^\infty(\mbG)\bar\tens N$ be an ergodic coaction of a locally compact quantum group $\mbG$. We say that $(N,\Delta_N)$ is a quantum homogeneous space (QHS) if there exists a $\mbG$ -$\C^*$-algebra $\mbD$ such that:
\begin{itemize}
\item $\C_0(\mbD)$ is a strongly dense $\C^*$-subalgebra of $N$;
\item $\Delta_{\mbD}$ is given by the restriction of $\Delta_N$ to $\C_0(\mbD)$;
\item $\Delta_N(N)\subset\M(\cml\otimes\C_0(\mbD))$ and the map $\Delta_N:N\rightarrow \mathcal{L}(\W^2(\mbG)\otimes \C_0(\mbD))$ is strict.
\end{itemize} 
\end{dfn}
\begin{rem}\label{weakcond}
It turns out that the form of Definition \ref{qhs2} is not optimal. Indeed, one may replace the last condition by an apparently weaker condition:
\begin{itemize}
\item For any uniformly bounded net $x_i\in N$ converging $*$-strongly to $x\in N$ and for any $y\in\cml\otimes\C_0(\mbD)$, the net $\Delta_N(x_i)y$ converges in the norm topology to $\Delta_N(x)y$.
\end{itemize} 
In order to see that the above condition implies the third condition of Definition \ref{qhs2} let us fix $x\in N$ and a uniformly bounded net $d_i\in\C_0(\mbD)$ that $*$-strongly converges to $x$. Then $\Delta_{\mbD}(d_i)\in\M(\C_0(\mbG)\otimes\C_0(\mbD))\subset\M(\cml\otimes\C_0(\mbD))$ and the above condition implies that $\Delta_{\mbD}(d_i)$ converges to $\Delta_N(x)$ in the strict topology of $\M(\cml\otimes D)$. In particular  $\Delta_N(x)\in\M(\cml\otimes\C_0(\mbD))$ and $\Delta_N$ gives rise to a strict map $\Delta_N:N\rightarrow\mathcal{L}(\W^2(\mbG)\otimes\C_0(\mbD))$.
\end{rem}
\begin{rem}
Let $\mbG_1$ be a closed quantum subgroup of $\mbG$ in the sense of Definition 2.5 of \cite{Vaes}. Using the results of Section 4 of \cite{Vaes} we may construct the von Neumann version of a  quantum quotient space $Q\subset \W^\infty(\mbG)$.
Definition \ref{qhs2} was motivated by Theorem 6.1 of \cite{Vaes} where it was shown that under the regularity  assumption imposed on $\mbG$ there exists a unique $\C^*$-algebraic version $\mbG/\mbG_0$ of the quantum quotient  satisfying the conditions of Definition \ref{qhs2} (for the notion of regularity we refer to \cite{BS}). The uniqueness of $\mbG/\mbG_0$ does not require the regularity of $\mbG$ which was already noted by S. Vaes in the course of his proof. This suggests that the uniqueness of $\mbD$ should hold in the context of Definition \ref{qhs2} which we shall prove in the next proposition. As we  show in \cite{pksolEQHS}, in the case of the quantum quotient by a compact quantum subgroup also the existence of $\mbG/\mbG_0$ does not require  the regularity of $\mbG$. For the $\C^*$-algebraic account of this construction we refer to \cite{SolQC}.
\end{rem}
\begin{prp}\label{emb1} Let $(N,\Delta_N)$ be a quantum homogeneous space in the sense of  Definition \ref{qhs2}. Then the $\mbG$-$\C^*$-algebra $\mbD$ is uniquely determined by the conditions of Definition \ref{qhs2}.
\end{prp}
\begin{proof} The below proof is a straightforward generalization of the Vaes' uniqueness argument - see  Theorem 6.1 of \cite{Vaes}.
Suppose that $\mbD_1$ and $\mbD_2$ satisfy the conditions of Definition \ref{qhs2}. Using the continuity of the coaction $\Delta_{\mbD_1}$ and 	the strictness of $\Delta_N$ we get
\begin{eqnarray*}\C_0(\mbD_1)&=&[\C_0(\mbD_1)\C_0(\mbD_1)]\\&=&[(\omega\otimes\iota)(\Delta_{\mbD_1}(\C_0(\mbD_1))(\cml\otimes \C_0(\mbD_1)))\,|\,\omega\in\B(\W^2(\mbG))_*]\\&=&[(\omega\otimes\iota)(\Delta_{N}(N)(\cml\otimes \C_0(\mbD_1)))\,|\,\omega\in\B(\W^2(\mbG))_*]
\\&=&[(\omega\otimes\iota)(\Delta_{N}(\C_0(\mbD_2))(\cml\otimes \C_0(\mbD_1)))\,|\,\omega\in\B(\W^2(\mbG))_*]=[\C_0(\mbD_2)\C_0(\mbD_1)].
\end{eqnarray*}
In the forth equality we used the strong density of $\C_0(\mbD_2)\subset N$ and again the strictness of $\Delta_N$. 
By symmetry we have $\C_0(\mbD_2)=[\C_0(\mbD_1)\C_0(\mbD_2)]$. Taking the adjoint we get $\C_0(\mbD_1)=\C_0(\mbD_2)$. 
\end{proof}
Let $(N,\Delta_N)$ be a QHS.
Our next aim is to prove that the $\C^*$-algebraic version $\mbD$ uniquely determines  $(N,\Delta_N)$. This justifies the following notation: $N=\W^\infty(\mbD)$, $\Delta_{N}=\Delta_{\W^\infty(\mbD)}$.
\begin{prp} Let $\mbG$ be a locally compact quantum group and let $(N_1,\Delta_{N_1})$ and $(N_2,\Delta_{N_2})$ be quantum homogeneous $\mbG$-spaces. Suppose that there exists a $\mbG$-isomorphism $\pi\in\Mor_{\mbG}(\mbD_1,\mbD_2)$. Then $\pi$  extends to a $\mbG$-isomorphism $\pi:N_1\rightarrow N_2$, $\Delta_{N_1}\circ\pi=(\iota\otimes\pi)\circ\Delta_{N_2}$. 
\end{prp}  
\begin{proof}
Using the fact that $\Delta_{N_1}(N_1)\subset\M(\cml\otimes\C_0(\mbD_1))$ we may define an injective map $(\iota\otimes\pi)\circ\Delta_{N_1}:N_1\rightarrow\M(\cml\otimes\C_0(\mbD_2))$. The $\mbG$-covariance of $\pi$ and the strictness of $\Delta_{N_1}$ and $\Delta_{N_2}$ implies that $(\iota\otimes\pi)\circ\Delta_{N_1}(N_1)=\Delta_{N_2}(N_2)$. Let $U_{N_2}$ be the canonical implementation of $\Delta_{N_2}$ (see Theorem \ref{unimpth}). The extension $\pi:N_1\rightarrow N_2$ of $\pi\in\Mor_{\mbG}(\mbD_1,\mbD_2)$ is given by the following explicit formula: $1\otimes \pi(x)=U^*_{N_2}\big((\iota\otimes\pi)\Delta_{N_1}(x)\big)U_{N_2}$ for any $x\in N_1$. 
\end{proof}
\end{section}
\begin{section}{$\mbG$-simplicity}\label{minsec}
Let $\G$ be a locally compact group and let $\X$ be a $\G$-space. The action of $\G$ on $\X$ is said to be minimal if $\overline{\G x}=\X$ for any $x\in\X$. In other words, the action of $\G$ on $\X$ is minimal if there does not exist a non-trivial  $\G$-invariant closed subset of $\X$. This condition  may be translated into the non-existence of a non-trivial $\G$-invariant ideal in $\C_0(\X)$ which leads  to the following notion of a $\mbG$-simple $\C^*$-algebra (see Definition 2.4 \cite{pkrdhs}).
\begin{dfn}\label{mincoact} Let $\mbG$ be a locally compact quantum group and $\mbD$ a $\mbG$-$\C^*$-algebra. We say that $\mbD$ is $\mbG$-simple if for any $\mbG$-$\C^*$-algebra  $\mathbb{B}$ and any $\mbG$-morphism $\pi\in\Mor_{\mbG}(\mbD,\mbB)$ we have $\ker\pi=\{0\}$.  
\end{dfn} 
In the next theorem we shall provide the quantum counterpart of the  classically trivial implication: transitivity $\Rightarrow$ minimality. It may be surprising that the proof relies on the Tomita-Takesaki theory employed in the construction of the canonical implementation of the coaction. 
\begin{thm}\label{minim} Let $\mbG$ be a locally compact quantum group and let $(N,\Delta_N)$ be a QHS in the sense of Definition \ref{qhs2}. The $\C^*$-algebraic version $\mbD$ of $(N,\Delta_N)$ is $\mbG$-simple.
\end{thm}
\begin{proof}
Let $\pi\in\Mor_{\mbG}(\mbD,\mbB)$ and let $I_0$ denote the kernel of $\pi$. Let $I\subset N$ be the strong closure of $I_0$. The strictness of $\Delta_N:N\rightarrow\mathcal{L}(\W^2(\mbG)\otimes \C_0(\mbD))$ implies that for any $x\in I$ and any $y\in\cml\otimes\C_0(\mbD)$ we have $(\iota\otimes\pi)[\Delta(x)y]=0$. Using the nuclearity of $\cml$ we see that $\Delta(x)y\in\cml\otimes I_0$. Replacing $y$ by the elements $e_i$ of  approximate unit for $\cml\otimes I_0$ and taking the limit we see that \begin{equation}\label{inI}\Delta(x)\in B(\W^2(\mbG))\bar\tens I.\end{equation} Let $p\in N$ be the central projection in $N$ corresponding to $I$, $I=pN$ (see Remark \ref{2id}). Note that \eqref{inI} implies the following inequality \begin{equation}\label{obeq}\Delta_N(p)\leq 1\otimes p.\end{equation} In order to prove the opposite inequality let us use the canonical implementation of $\Delta_N$. Adopting the notation of Theorem \ref{unimpth} we express  \eqref{obeq} by  \[U(1\otimes p)U^*\leq 1\otimes p.\] This  easily implies that $(1\otimes p)\leq U^*(1\otimes p)U$. Conjugating both sides with $\hat{J}\otimes J_\theta$ and using the centrality of $p$ we get  $(1\otimes p)\leq U(1\otimes p)U^*=\Delta_N(p)$ which together with \eqref{obeq} gives rise to the equality  
\begin{equation}\label{invp}\Delta_N(p)=1\otimes p.\end{equation} 
Using the ergodicity of $\Delta_N$ we get $p=0$ or $p=1$.
In the case $p=0$ we have $\ker\pi=\{0\}$ which proves the $\mbG$-simplicity of $\mbD$. 

The case $p=1$  can be translated into the strong density of  $I_0$ inside $N$. Repeating the argument from the first paragraph of the proof we may see that
\begin{equation}\label{i0c}[\Delta_{\mbD}(I_0)(\cml\otimes\C_0(\mbD))]\subset\cml\otimes I_0.\end{equation} Let us fix an element $d\in\C_0(\mbD)$ and $y\in\cml\otimes\C_0(\mbD)$. There exists a uniformly bounded net $x_i\in I_0$ that $*$-strongly converges to $d$. Using the strictness condition for $\Delta_N$ and inclusion \eqref{i0c} we get the norm convergence
\[\Delta_{\mbD}(d)y=\lim_i\Delta_{\mbD}(x_i)y.\] This together with  \eqref{i0c} shows that 
\[\cml\otimes\C_0(\mbD)=[\Delta_{\mbD}(\C_0(\mbD))(\cml\otimes\C_0(\mbD))]\subset\cml\otimes I_0\] and immediately gives $\C_0(\mbD)\subset I_0$. Since the opposite inclusion is clear we get $\ker\pi=\C_0(\mbD)$ which contradicts the non-degenerateness of $\pi$ and rules out the case $p=1$.
\end{proof}
\end{section}

\begin{section}{Classical case}\label{seccl}
Let $\mathrm{G}$ be a locally compact group and $\mbG$ the corresponding   locally compact quantum group: $M=\W^\infty(\mathrm{G})$ and $\Delta:\W^\infty(\mathrm{G})\rightarrow\W^\infty(\mathrm{G})\bar\otimes\W^\infty(\mathrm{G})$ is the standard comultiplication.
Let $\mathrm{X}$ be a homogeneous $\mathrm{G}$-space. The action of $\mathrm{G}$ on $\mathrm{X}$ gives rise to the coaction of $\mbG$ on $\W^\infty(\mathrm{X})$: $\Delta_{\mathrm{X}}:\W^\infty(\mathrm{X})\rightarrow\W^\infty(\mathrm{G})\bar\tens\W^\infty(\mathrm{X})$, $\Delta_{\mathrm{X}}(f)(g,x)=f(gx)$ for any $g\in\mathrm{G}$, $x\in\mathrm{X}$ and $f\in\W^\infty(\mathrm{X})$. It turns out that $(\W^\infty(\mathrm{X}),\Delta_{\mathrm{X}})$ is a QHS in the sense of Definition \ref{qhs2} with the $\C^*$-algebraic version $\mbX=(\C_0(\mathrm{X}),\Delta_{\mathrm{X}})$. The fact that $\mbX$ satisfies all the conditions of Definition \ref{qhs2} may be concluded from Theorem 6.1 \cite{Vaes}. Our aim is to prove the opposite implication which together with the above provides  a 1-1 correspondence between the quantum homogeneous $\mbG$-spaces with the underlying commutative von Neumann algebra $N$ on one side and homogeneous $\mathrm{G}$-spaces on the other. 
\begin{thm}\label{clehs} Let $\mathrm{G}$ be a locally compact group and $\mbG$ the corresponding LCQG. Let  $(N,\Delta_N)$ be a quantum homogeneous space in the sense of Definition \ref{qhs2} with $N$ being commutative. Then the spectrum $\mathrm{X}=\Sp(\mbD)$ is a homogeneous $\mathrm{G}$-space and $(N,\Delta_N)=(\W^\infty(\mathrm{X}),\Delta_{\mathrm{X}})$.
\end{thm}
\begin{proof}
Let us fix a point $\tau\in\X$ and let $\ev_\tau:\C_0(\mbD)\rightarrow\mathbb{C}$ be the corresponding character. Let us consider the morphism $\pi_{\mbD}\in\Mor_{\mbG}(\mbD,\mbG)$ given by the formula $\pi_{\mbD}(d)=(\iota\otimes\ev_\tau)\Delta_{\mbD}(d)$ for any $d\in\C_0(\mbD)$. Using the $\mbG$-simplicity of  $\mbD$ (see Theorem \ref{minim}) we see that $\pi_{\mbD}$ is injective.

In what follows we shall show that $\pi_{\mbD}$ may be extended to an injective normal $*$-homomorphism of $\pi_{N}:N\rightarrow\W^\infty(\mbG)$. The strictness of $\Delta_N$ gives $\Delta_N(N)\subset\M(\cml\otimes\C_0(\mbD))$, which enables us to define $\pi_N(x)=(\iota\otimes\ev_\tau)\Delta_N(x)$ for any $x\in N$. Let $I=\ker\pi_N$ and let $p\in N$ be the central projection generating $I$: $I=pN$. Using the equality  
\begin{equation}\label{cover}(\iota\otimes\pi_N)\circ\Delta_N=\Delta_{\mbG}\circ\pi_N\end{equation} we may see that $\Delta_N(I)\subset \W^\infty(\mbG)\bar\tens I$. In particular $\Delta_N(p)\leq 1\otimes p$. Proceeding as in the proof of Theorem \ref{minim} we get the equality $\Delta_N(p)=1\otimes p$. The ergodicity of $\Delta_N$ implies that either $p=0$ or $p=1$. The case $p=1$ is ruled out by the injectivity of $\pi_{\mbD}$ hence $p=0$ and $\pi_N$ is injective. Identifying $N$ with its image in $\W^\infty(\mbG)$ under  $\pi_N$ end using \eqref{cover} we get   $\Delta_N = \Delta_{\mbG}|_N$, thus $N\subset \W^\infty(\mbG)$ is a coideal:
$\Delta_\mbG(N)\subset\W^\infty(\mbG)\bar\otimes N$. Using Theorem 2 \cite{TT} we get a closed subgroup $\mathrm{H}\subset \mathrm{G}$ such that $N=\W^\infty(\mathrm{G}/\mathrm{H})$. Thus Theorem 6.1 \cite{Vaes} enables us to identify $\Sp(\mbD)$ with $\mathrm{G}/\mathrm{H}$. 
\end{proof}
\end{section}
\begin{section}{Quantum homogeneous spaces of compact quantum groups}\label{secqhsp}
The aim of this section is to show that our definition of a  quantum homogeneous space is compatible with the the definition of quantum homogeneous space of a compact quantum group introduced by P. Podle\'s \cite{Pod1}. Let us first recall the notion of a compact quantum group (see \cite{Worcqg}).
\begin{dfn}\label{cqg}
A compact quantum group $\mbG=(A,\Delta)$ consists of a unital $\C^*$-algebra $A$ together with a unital $*$-homomorphism $\Delta:A\rightarrow A\otimes A$ satisfying the coassociativity relation \[(\Delta\otimes\iota)\circ\Delta=(\iota\otimes\Delta)\circ\Delta\] and the cancellation properties
\[[\Delta(A)(A\otimes 1)]=A\otimes A=[\Delta(A)(1\otimes A)].\]
\end{dfn}
If $\mbG$ is a compact quantum group then there exists a unique Haar state $\phi\in A^*$ on it:
\[(\phi\otimes\iota)\Delta(a)=\phi(a)1=(\iota\otimes\phi)\Delta(a).\]
In what follows we shall assume that $\phi$ is faithful. 
\begin{dfn} Let $\mathcal{H}$ be a Hilbert space. A unitary corepresentation $u$ of $\mbG$ on $\mathcal{H}$ is a unitary element of $\M(\C_0(\mbG)\otimes\mathcal{K}(\mathcal{H}))$ such that $(\Delta\otimes\iota)u=u_{13}u_{23}$. The dimension of the underlying Hilbert space $\mathcal{H}$ is called the dimension of $u$ and denoted by $\dim u$. 
\end{dfn}The tensor product of the unitary corepresentation $u$ and $v$ is defined by $u\tp v:=u_{12}v_{13}$.
As was proved by S.L. Woronowicz, a unitary corepresentation may be decomposed into a direct sum of the irreducible ones and the irreducible corepresentations are finite dimensional.
In the course of this section we shall write $\widehat\mbG$ for the set of equivalence classes of irreducible corepresentations of $\mbG$.  For all $\alpha\in\widehat\mbG$ we choose  unitary representatives $u^\alpha\in\C_0(\mbG)\otimes\B(\mathcal{H}^\alpha)$. 
\begin{dfn}\label{cqhs} Let $\mbG$ be a compact quantum group and $\mbD$ a unital $\mbG$-$\C^*$-algebra. We say that $\mathbb{D}$ is a quantum homogeneous space in the sense of Podle\'s (QHSP) if $\mbD$ is ergodic: \[\{d\in\C_0(\mbD)\,|\,\Delta_{\mbD}(d)=1\otimes d\}=\mathbb{C}\cdot 1.\]
\end{dfn} 
\begin{ntn} Let $x,y\in\W^2(\mbG)$. We define $\omega_{x,y}\in\B(\W^2(\mbG))_*$ by the formula
$\omega_{x,y}(T)=(x|Ty)$ for any $T\in\B(\W^2(\mbG))$. For any $x,y\in \C_0(\mbG)\subset\W^2(\mbG)$ (note the identification of $\C_0(\mbG)$ with its image under the GNS-map $\Lambda_\phi$) and any $a\in \C_0(\mbG)$ we have $\omega_{x,y}(a)=\phi(x^*ay)$. In this context we shall use the notation $\omega_{x,y}(\cdot)=\phi(x^*\cdot y)$.
\end{ntn}
Let $\mbD$ be a QHSP and let us fix an irreducible corepresentation $\alpha\in\widehat\mbG$.  
Let $W^\alpha_{\mbD}\subset\C_0(\mbD)$ be the spectral subspace corresponding to $\alpha$. Choosing a basis in $\mathcal{H}^\alpha$ we may identify $u^\alpha$ with the matrix of elements $u^\alpha_{ij}\in\C_0(\mbG)$, $i,j=1,\ldots,\dim\alpha$. Then $W^\alpha_{\mbD}=\mbox{span}\{(\phi(u^{\alpha*}_{ij}\,\cdot)\otimes\iota)\Delta_{\mbD}(d)\,|\,d\in\C_0(\mbD),\,\,i,j=1,\ldots,\dim\alpha\}$. For a canonical definition  of $W^\alpha_{\mbD}$ we refer to Definition 2.3 of \cite{VaesErg} (see also \cite{BF} and \cite{Pod1}).  It turns out that the  $\dim W^\alpha_{\mbD}<\infty$. The precise estimation of $\dim W^\alpha_{\mbD}$ in terms of the quantum dimension  is the subject of Theorem 17 of \cite{BF}  and Theorem 2.5 of \cite{VaesErg}. We shall write $\mathcal{D}\subset\C_0(\mbD)$ for the subspace generated by the spectral subspaces: \begin{equation}\label{sme}\mathcal{D}=\oplus_{\alpha\in\widehat\mbG}W^\alpha_{\mbD}.\end{equation}
Treating $\mbG$ as a $\mbG$-$\C^*$-algebra we may consider the spectral subspaces $W^\alpha_{\mbG}$. The following finite dimensionality result will be used in the main theorem of this section.
\begin{lem}\label{pwt}
Let $\alpha\in\widehat\mbG$ be an irreducible corepresentation of $\mbG$, $y\in W^\alpha_\mbG$ and let $\mbD$ be a QHSP. Then $\{(\phi(y^*\cdot y)\otimes\iota)\Delta_{\mbD}(d)\,|\,d\in\C_0(\mbD)\}$ is a finite dimensional subspace of $\C_0(\mbD)$. 
\end{lem}
\begin{proof} 
Let $\sigma:\mathbb{R}\rightarrow\Aut(A)$ be the KMS-automorphism of the Haar state (see Theorem 1.4, \cite{Worcqg}). Note that
\begin{equation}\label{slice}(\phi(y^*\cdot y)\otimes\iota)\Delta_{\mbD}(d)
=(\phi((y\sigma_{-i}(y^*))^*\cdot 1)\otimes\iota)\Delta_{\mbD}(d).\end{equation}
Let $\alpha_1,\ldots\alpha_n$ be the irreducible corepresentations entering the decomposition of $\alpha\tp\alpha^*$ onto the irreducible components. Using Equation \eqref{slice} and the fact that $\sigma_{-i}(y^*)\in W^{\alpha^*}_{\mbG}$ we get $(\phi(y^*\cdot y)\otimes\iota)\Delta_{\mbD}(d)\in W^{\alpha_1}_{\mbD}\oplus\ldots\oplus W^{\alpha_n}_{\mbD}$. The finite dimensionality of $W^{\alpha_i}_{\mbD}$ implies that  $\dim\,\{(\phi(y^*\cdot y)\otimes\iota)\Delta_{\mbD}(d)\,|\,d\in\C_0(\mbD)\}<\infty$
\end{proof}
Let us move on to the von Neumann version of $\mbD$. For the details of the following standard construction we refer to \cite{BF} or \cite{Wang}.  The ergodicity of $\Delta_{\mbD}$ implies that $(\phi\otimes\iota)\Delta_{\mbD}(d)\in\mathbb{C}\cdot 1$. 
This enables us to define a state $\rho:\C_0(\mbD)\rightarrow\mathbb{C}$: $(\phi\otimes\iota)\Delta_{\mbD}(d)=\rho(d)1$. Using the injectivity of $\Delta_{\mbD}$ and the faithfulness of $\phi$ we may see that  $\rho$ is faithful. In order to check the $\mbD$-invariance of  $\rho$ we compute: 
\begin{eqnarray*}
(\iota\otimes\rho)\Delta_{\mbD}(d)&=&(\phi\otimes\iota\otimes\iota)(\iota\otimes\Delta_{\mbD})\Delta_{\mbD}(d)\\
&=&(\phi\otimes\iota\otimes\iota)(\Delta_{\mbG}\otimes\iota)\Delta_{\mbD}(d)\\
&=&(\phi\otimes\iota)\Delta_{\mbD}(d)=\rho(d)1.
\end{eqnarray*}
Let $(\W^2(\mbD),\pi_{\rho},\Lambda_{\rho})$ be the GNS-construction corresponding to $\rho$, i.e. $\pi_{\rho}:\C_0(\mbD)\rightarrow\B(\W^2(\mbD))$ is the GNS-representation and $\Lambda_{\rho}:\C_0(\mbD)\rightarrow\W^2(\mbD)$ is the GNS-map. There exists a unitary  operator $U_{\mbD}\in\B(\W^2(\mbG)\otimes\W^2(\mbD))$ such that $U_{\mbD}(\Lambda_{\phi}(a)\otimes\Lambda_{\rho}(d))=(\Lambda_{\phi}\otimes\Lambda_{\rho})(\Delta_{\mbD}(d)(a\otimes 1))$. In order to see that $U_{\mbD}$ is isometric one uses the $\mbD$-invariance of $\rho$. The fact that $U_{\mbD}$ is surjective follows from the continuity of $\Delta_{\mbD}$: $[\Delta_{\mbD}(\C_0(\mbD))(\C_0(\mbG)\otimes 1)]=\C_0(\mbG)\otimes\C_0(\mbD)$. 
In what follows we shall treat $\C_0(\mbG)$ and $\C_0(\mbD)$ as concrete $\C^*$-algebras acting on the Hilbert spaces $\W^2(\mbG)$ and $\W^2(\mbD)$ respectively. On the notational level it means that we shall skip $\pi_\rho$ and $\pi_\phi$ from all formula. 

 Let $N\subset\B(\W^2(\mbD))$ be the strong closure of $\C_0(\mbD)$. Noting that $\Delta_{\mbD}(d)=U_{\mbD}(1\otimes d)U_{\mbD}^*$ we may extend $\Delta_{\mbD}$ to $\Delta_N:N\rightarrow \W^\infty(\mbG)\bar\tens N$: $\Delta_N(x)=U_{\mbD}(1\otimes x)U_{\mbD}^*$ for any $x\in N$. In what follows we shall refer to $(N,\Delta_N)$ as a von Neumann version of $\mbD$. 
Before we move on to the proof that $(N,\Delta_N)$ is a QHS in the sense of Definition \ref{qhs2} let us extend  Lemma \ref{pwt} to the pair $(N,\Delta_N)$.
\begin{lem}\label{pwt1}
Let $(N,\Delta_N)$ be the von Neumann version of the quantum homogeneous space $\mbD$. Let $\alpha\in\widehat\mbG$ be an irreducible representation and let $y\in W^\alpha_{\mbG}$. Then $\{(\phi(y^*\cdot y)\otimes\iota)\Delta_N(x)|\,x\in N\}$ is finite dimensional. 
\end{lem}
\begin{proof}
Let $x\in N$ and let $d_i\in\C_0(\mbD)$ be a net that strongly converges to $x$. Using the notation of Lemma \ref{pwt} we may see that  $x_i\equiv(\phi(y^*\cdot y)\otimes\iota)\Delta_N(d_i)\in W^{\alpha_1}_{\mbD}\oplus\ldots\oplus W^{\alpha_n}_{\mbD}$. The strong convergence of the net $x_i$ and the fact that the subspace $W^{\alpha_1}_{\mbD}\oplus\ldots\oplus W^{\alpha_n}_{\mbD}$ is closed in the strong topology (this follows from its finite dimensionality) imply that $x\in W^{\alpha_1}_{\mbD}\oplus\ldots\oplus W^{\alpha_n}_{\mbD}$. In particular $\dim \{(\phi(y^*\cdot y)\otimes\iota)\Delta_N(x)|\,x\in N\}<\infty $.
\end{proof}
\begin{thm}\label{cqhs1} Let $\mbG$ be a compact quantum group with a faithful Haar state $\phi$ and let $\mbD$ be a quantum homogeneous space of $\mbG$ in the sense of Podle\'s. Let $(N,\Delta_N)$ be  the von Neumann version of $\mbD$ described above. Then $(N,\Delta_N)$ is a quantum homogeneous space in the sense of Definition \ref{qhs2} and its $\C^*$-algebraic version coincides with $\mbD$.
\end{thm} 
\begin{proof}
Let us first extend  $\rho$ to a n.s.f. state on $N$. Using the extension of $\phi$ to $\W^\infty(\mbG)$ we define $\rho(x)1=(\phi\otimes\iota)\Delta_N(x)$.  Note that  $\rho:N\rightarrow\mathbb{C}$ is  $\Delta_N$-invariant.  Suppose that $\Delta_N(x)=1\otimes x$ for some $x\in N$. The computation \[\rho(x)1=(\phi\otimes\iota)\Delta_N(x)=(\phi\otimes\iota)(1\otimes x)=x\] shows that $\Delta_N$ is ergodic. 

It is clear that $\C_0(\mbD)$ is strongly dense in $N$ and that the coaction $\Delta_N$ when restricted to $\C_0(\mbD)$ coincides with $\Delta_{\mbD}$. 
In order to prove the strictness of $\Delta_N$ let us consider a uniformly bounded net $x_i\in N$, $||x_i||\leq C$ converging  to $0$ in the $*$-strong topology. Note that $x_i^*x_i\in N$ $*$-strongly converges to $0$. Indeed,  $||x_i^*x_iv||\leq  C||x_iv||\rightarrow 0$ for any $v\in\W^2(\mbD)$.
Let $\alpha\in\widehat{\mbG}$ be an irreducible corepresentation and let us fix  $y\in W^\alpha_{\mbG}$. Let $P_y$ be the projection onto the $1$-dimensional subspace spanned by $y\in\W^2(\mbG)$. Using the uniform boundness of $x_i$ and the unitality of $\C_0(\mbD)$ one may see that in order to prove the strictness of $\Delta_N$ it is enough to show that  $\Delta_N(x_i)(P_y\otimes 1)$ converges to zero in the norm sense. Indeed, using the fact that $x_i$ is uniformly bounded we see that the norm convergence of $\Delta_N(x_i)(P_y\otimes 1)$ implies the norm convergence of $\Delta_N(x_i)(k\otimes 1)$ for any $k\in\cml$. This in turn implies  that for any $b\in\cml\otimes\C_0(\mbD)$ the net $\Delta_N(x_i)b$ norm converges to 0. Finally, in order to see that for any $x\in N$ we have $\Delta_N(x)\in\M(\cml\otimes\C_0(\mbD))$ let us fix a uniformly bounded net $d_i\in\C_0(\mbD)$ that  $*$-strongly converges to $x$. The above argument shows that $\Delta_{\mbD}(d_i)b$ norm converges to $\Delta_N(x)b$ which together with Remark \ref{weakcond} shows that $\Delta_N(x)\in\M(\cml\otimes\C_0(\mbD))$. 

Let us move on to the proof of the fact that $\Delta_{N}(x_i)(P_y\otimes 1)$ norm converges to $0$. We compute:
\begin{eqnarray*}
||\Delta_{N}(x_i)(P_y\otimes 1)||^2&=&||(P_y\otimes 1)\Delta_{N}(x_i^*x_i)(P_y\otimes 1)||\\
&=&||(\phi(y^*\cdot\,y)\otimes\iota)\Delta_{N}(x_i^*x_i)||.
\end{eqnarray*}
The finite dimensionality of the subspace $\{(\phi(y^*\cdot\,y)\otimes\iota)\Delta_{N}(x_i^*x_i)|\,x\in N\}$ proved in Lemma \ref{pwt1} and the fact that $(\phi(y^*\cdot\,y)\otimes\iota)\Delta_{N}(x_i^*x_i)$  converges strongly to zero imply the norm convergence  \begin{equation*}\lim_i\,(\phi(y^*\cdot\,y)\otimes\iota)\Delta_{N}(x_i^*x_i)=0.\qedhere\end{equation*}
\end{proof}
In the  next theorem we shall prove that an ergodic coaction of a compact quantum group on a von Neumann algebra is automatically a QHS in the sense of Definition \ref{qhs2}.
\begin{thm} 
Let $\mbG$ be a compact quantum group and let $\Delta_N:N\rightarrow\W^\infty(\mbG)\bar{\tens} N$ be an ergodic coaction of $\mbG$ on a von Neumann algebra $N$. Then $(N,\Delta_N)$ is a QHS in the sense of Definition \ref{qhs2}.
\end{thm}
\begin{proof} Let us only sketch  the proof. Extending the definition of the spectral subspaces to the von Neumann setting we may define $\mathcal{D}\subset N$ - the direct sum of the spectral subspaces of $\Delta_N$ (compare with \eqref{sme}).
It may be checked that the $\C^*$-completion of $\mathcal{D}$ gives rise to  a $\mbG$-$\C^*$-algebra which we shall denote by $\mbD$. Using the techniques of the proof of Theorem \ref{cqhs1} one may see that  $\mbD$ satisfies all conditions of Definition \ref{qhs2}. 
\end{proof}
\begin{rem} The  above theorem shows that $\mbD$ is a unital $\mbG$-$\C^*$-algebra, where the unit is provided by the trivial representation entering $N$ under the form of $\mathbb{C}1\subset N$. This observation rules out  from our framework the paradoxical examples of non-compact quantum homogeneous spaces of a compact quantum group given  in \cite{woreq2}. 
\end{rem}
\end{section}
\begin{section}{Rieffel deformation of homogeneous spaces}\label{secrd}
Rieffel deformation is a well established method of deforming $\C^*$-algebras. In his original approach  M. Rieffel starts from  deformation data $(A,\rho,J)$ which consists of a  $\C^*$-algebra $A$, an action $\rho$ of $\mathbb{R}^n$  on $A$ and a skew symmetric matrix $J:\mathbb{R}^n\rightarrow\mathbb{R}^n$. Using these data M. Rieffel was able to deform the original product on the algebra of $\rho$-smooth elements: $A^\infty \subset A$. The deformed $\C^*$-algebra $A^J$ is defined as a $\C^*$-algebraic completion of $A^\infty$ considered as an algebra with this deformed product. The Rieffel deformation which was originally developed to deform $\C^*$-algebras (see \cite{Rf1}) may then be applied for the deformation of locally compact quantum groups (see \cite{Rf2}).
\begin{subsection}{Rieffel deformation via crossed products}\label{rdcp}
In our recent approach to the Rieffel deformation  (see \cite{pkrd})  we formulate the deformation  procedure in terms of the crossed product structure. The starting point is the choice of the deformation data $(A,\rho,\Psi)$ where $A$ is a  $\C^*$-algebra, $\rho$ is a continuous action  of an abelian group $\Gamma$ on $A$ and $\Psi$ is a  $2$-cocycle  on the dual group $\Hat\Gamma$.
We shall only consider continuous unitary $2$-cocycles, i.e.  continuous functions
  $\Psi:\hat{\Gamma}\times
\hat{\Gamma}\rightarrow
 \mathbb{T}$ satisfying:
\begin{itemize}
\item[(i)] $\Psi(e,\hat\gamma)=\Psi(\hat\gamma,e)=1$ for all $\hat\gamma\in\Hat\Gamma$;
\item[(ii)]$\Psi(\hat{\gamma}_1,\hat{\gamma}_2+\hat{\gamma}_3)
\Psi(\hat{\gamma}_2,\hat{\gamma}_3)
=\Psi(\hat{\gamma}_1+\hat{\gamma}_2,\hat{\gamma}_3)\Psi(
\hat{\gamma}_1,\hat{\gamma}_2) $ for all
$\hat{\gamma}_1,\hat{\gamma}_2,\hat{\gamma}_3\in\Hat\Gamma$.
\end{itemize}
For the theory of $2$-cocycles we refer to \cite{kl}.

 The deformation procedure of $A$ consists of the following steps:\\
1. Let $B$ be the  crossed product $\C^*$-algebra $B=\Gamma\ltimes_\rho A$ and  let $(B,\lambda,\hat\rho)$ be the $\Gamma$-product structure on this  crossed product i.e. $\lambda:\Gamma\rightarrow\M(B)$ is the representation of $\Gamma$ implementing the action $\rho$ and $\hat\rho$ is the dual action on $B$.\\
2. Let $\lambda\in\Mor(\C^*(\Gamma),B)$ be the morphism corresponding to the representation $\lambda\in\Rep(\Gamma,B)$ and let $\Psi_{\hg}\in\M(\C^*(\HG))$ be the family of functions given by $\Psi_{\hg}(\hg')=\Psi(\hg',\hg)$. Applying $\lambda$ to $\Psi_{\hg}$ (note the identification of $\C_0(\HG)$ with $\C^*(\Gamma)$ via the Fourier transform), we get a $\hat\rho$-projective unitary 1-cocycle  $U_{\hg}\in\M(B)$:\[U_{\hat\gamma_1+\hat\gamma_2}=\overline{\Psi(\hat\gamma_1,\hat\gamma_2)}U_{\hat\gamma_1}\hat\rho_{\gamma_1}(U_{\hat\gamma_2}).\] We define the deformed dual action $\hat\rho^\Psi:\HG\rightarrow\Aut(B)$ by the formula: $\hat\rho_{\hg}(b)=U_{\hg}^*\hat\rho_{\hg}(b)U_{\hg}$,\,\,for any $\hg\in\HG$ and $b\in B$.\\
3. The deformed $\C^*$-algebra $A^\Psi$ is defined as the Landstad algebra of the deformed $\Gamma$-product $(B,\lambda,\hat\rho^\Psi)$:
\[A^\Psi=\left\{b\in M(B)\left|\begin{array}{l}1.\,\,\hat\rho^\Psi_{\hg}(b)=b\\
2.\,\,\mbox{The map }\Gamma\ni\gamma\mapsto\lambda_\gamma a\lambda_\gamma^*\in \M(B)\\
\mbox{\,\,\,\,\,\, is norm-continuous}\\
3.\,\,\lambda(x)a,a\lambda(x)\in B \mbox{ for any }x\in \C^*(\Gamma)
\end{array}\right.\right\}\]
\end{subsection}
\begin{subsection}{Rieffel deformation of $\mbG$-$\C^*$-algebras}\label{rdlg} Let $\G$ be a locally compact group, $\mbG$ the corresponding LCQG (see Section \ref{seccl}) and 
let $\mbD$ be a $\mbG$-$\C^*$-algebra. 
As was shown in \cite{pkrdgc}, the Rieffel deformation can also be used for deforming $\mbD$ where we simultaneously deform $\mbG$, $\C_0(\mbD)$ and the coaction $\Delta_{\mbD}$. In what follows we shall give a concise description of the deformation procedure.

The deformation of $\mbG$ described below is the dual version of the $2$-cocycle twist of the comultiplication $\Delta_{\widehat\mbG}$  studied in \cite{V}  and then developed in \cite{DC} and \cite{FV}. The twisted coaction is given by the formula: $\Delta_{\widehat\mbG^\Psi}(x)=\Psi^*\Delta_{\widehat\mbG}(x)\Psi$ for any $x\in\C_0(\widehat\mbG)$. The $\C^*$-algebra of the deformed LCQG $\mbG^\Psi$ is constructed as follows. Let $\Gamma$ be a closed abelian subgroup of $\G$, $\Psi$  a continuous $2$-cocycle on $\HG$ and let $\rho:\Gamma^2\rightarrow\Aut(\C_0(\G))$ be the  continuous action given by the left and the right shifts along $\Gamma$. Let us consider a $2$-cocycle $\Psi\otimes\tilde\Psi$ on $\HG^2$ where $\tilde\Psi$ is defined by the formula: 
\begin{equation}\label{deftilpsi}\tilde{\Psi}(\hat\gamma_1,\hat\gamma_2)=\overline{\Psi(-\hat\gamma_1,-\hat\gamma_2)}.\end{equation} The deformation data $(\C_0(\G),\rho,\Psi\otimes\tilde\Psi)$  gives rise to the $\C^*$-algebra $\C_0(\mbG^\Psi)$, by means of the deformation procedure described in Section \ref{rdcp}. 

In order to construct the comultiplication $\Delta_{\mbG^\Psi}$ we employ the $\rho$-covariance properties of the the comultiplication $\Delta_{\mbG}$:
\begin{align} \label{cov1}(\rho_{e,\gamma}\otimes\rho_{\gamma,e})(\Delta_{\mbG}(f))&=\Delta_{\mbG}(f),\\
\label{cov}\Delta_{\mbG}(\rho_{\gamma,\gamma'}(f))&=(\rho_{\gamma,e}\otimes\rho_{e,\gamma'})(\Delta_{\mbG}(f)),
\end{align}
for all  $f\in\C_0(\G)$, $\gamma,\gamma'\in\Gamma$. 
Let $(B,\lambda,\hat\rho)$ be the standard $\Gamma^2$-product structure on $B=\Gamma^2\ltimes_\rho\C_0(\G)$. Using \eqref{cov} and the universal properties of the crossed product  we may extend $\Delta_{\mbG}$ to a morphism $\Delta^\Gamma_{\mbG}:B\rightarrow\M(B\otimes B)$. It is uniquely determined by the following two properties: 
\begin{itemize}
\item the action of $\Delta^\Gamma_{\mbG}$ on $\C_0(\mbG)\subset\M(B)$ consides with $\Delta_{\mbG}$;
\item the action of $\Delta^\Gamma_{\mbG}$ on $\lambda_{\gamma,\gamma'}\in\M(B)$ is given by $\Delta^\Gamma_{\mbG}(\lambda_{\gamma,\gamma'})=\lambda_{\gamma,e}\otimes\lambda_{e,\gamma'}$.
\end{itemize} Let us consider a unitary representation $U:\Gamma^2\rightarrow\M(B\otimes B)$
\[ U_{\gamma,\gamma'}=\lambda_{e,\gamma}\otimes\lambda_{\gamma',e}\] and a function $\Psi^{\star}\in\M(\C_0(\Hat\Gamma)\otimes\C_0(\Hat\Gamma))$ given by 
$\Psi^{\star}=\Psi(-\hat\gamma_1,\hat\gamma_1+\hat\gamma_2)$. Extending $U$ to the morphism $\pi_U:\C^*(\Gamma^2)\rightarrow\M(B\otimes B)$ we may apply it to $\Psi^\star$ obtaining \begin{equation}\label{defup}\Upsilon=\pi_U(\Psi).\end{equation}  Let us define $\Delta^\Gamma_{\mbG^\Psi}:B\rightarrow\M(B\otimes B)$
\begin{equation}\label{twstep}\Delta^\Gamma_{\mbG^\Psi}(b)=\Upsilon\Delta_{\mbG}^\Gamma(b)\Upsilon^*.\end{equation} Using \eqref{cov1} and \eqref{cov} one may show that $\Delta^\Gamma_{\mbG^\Psi}$ when restricted to $\C_0(\mbG^\Psi)\subset\M(B)$ defines the comultiplication: $\Delta_{\mbG^\Psi}:\C_0(\mbG^\Psi)\rightarrow\M(\C_0(\mbG^\Psi)\otimes\C_0(\mbG^\Psi))$. 

Suppose now that $\mbD$ is a $\mbG$-$\C^*$-algebra. Restricting the $\mbG$-coaction to the subgroup $\Gamma$ we get an action $\rho':\Gamma\rightarrow\C_0(\mbD)$. Performing the Rieffel deformation based on the deformation data $(\C_0(\mbD),\rho',\Psi)$  we get the $\C^*$-algebra $\C_0(\mbD^\Psi)$. In order to define $\Delta_{\mbD^\Psi}$ we first extend $\Delta_{\mbD}$ to a morphism
$\Delta_{\mbD}^\Gamma\in\Mor(\Gamma\ltimes_{\rho'}\C_0(\mbD),B\otimes\Gamma\ltimes_{\rho'}\C_0(\mbD))$.  In the analogy with \eqref{twstep} we may define $\Delta_{\mbD^\Psi}^\Gamma\in\Mor(\Gamma\ltimes_{\rho'}\C_0(\mbD),B\otimes\Gamma\ltimes_{\rho'}\C_0(\mbD))$. Restricting $\Delta_{\mbD^\Psi}^\Gamma$ to $\C_0(\mbD^\Psi)\subset\M(\Gamma\ltimes_{\rho'}\C_0(\mbD))$ we get the coaction $\Delta_{\mbD^\Psi}\in\Mor(\C_0(\mbD^\Psi),\C_0(\mbG^\Psi)\otimes\C_0(\mbD^\Psi))$. As was shown in \cite{pkrdgc} the above construction leads to a  $\mbG^\Psi$-$\C^*$-algebra $\mbD^\Psi$. Actually, the reader may have noticed some differences between the above description of the deformation procedure of $\mbD$ and the one given in \cite{pkrdgc}. They are due to the adopted duality relation between  $\mbG^\Psi$ and $\widehat\mbG^\Psi=(\C_0(\widehat\mbG),\Psi^*\Delta_{\widehat\mbG}(\cdot)\Psi)$ and the fact that $\mbD$ is a left $\mbG$-$\C^*$-algebra whereas in \cite{pkrdgc} we consider the right case.
\end{subsection}
\begin{subsection}{Rieffel deformation of homogeneous spaces}
Let $\X$ be a $\G$-homogeneous space and $\mbX$ the corresponding $\mbG$-$\C^*$-algebra (see Section \ref{seccl}). Let $\Gamma$ be a closed subgroup of $\G$ and $\Psi$ a $2$-cocycle on $\Hat\Gamma$. As was described in the previous section we may define the $\mbG^\Psi$-$\C^*$-algebra $\mbX^\Psi$. The aim of this section is to construct the von Neumann version $(\W^\infty(\mbX^\Psi),\Delta_{\W^\infty(\mbX^\Psi)})$ of $\mbX^\Psi$ and show that it is a QHS in the sense of Definition \ref{qhs2} with the corresponding $\C^*$-algebraic version coinciding with $\mbX^\Psi$.   
Let us move on to the construction of $(\W^\infty(\mbX^\Psi),\Delta_{\W^\infty(\mbX^\Psi)})$.

 As may be expected by the reader, $\W^\infty(\mbX^\Psi)$ is defined as a  strong closure of $\C_0(\mbX^\Psi)$. To be more precise we may take the strong closure  inside the von Neumann crossed product $\Gamma\ltimes_{\rho'}\W^\infty(\mbX)$.  
In order to keep the connection with the Rieffel deformation we need also  the description of $\W^\infty(\mbX^\Psi)$ in terms of the invariant element under the twisted dual action $\hat\rho'^\Psi:\Gamma\rightarrow\Aut(\Gamma\ltimes_{\rho'}\W^\infty(\mbX))$ (see Section \ref{rdcp}). For the later purposes we shall interpret the dual action $\hat\rho':\Gamma\rightarrow\Aut(\Gamma\ltimes_{\rho'}\W^\infty(\mbX))$ as a right coaction $\hat\rho':\Gamma\ltimes_{\rho'}\W^\infty(\mbX)\rightarrow\Gamma\ltimes_{\rho'}\W^\infty(\mbX)\bar{\tens}\W^\infty(\Hat\Gamma)$. Under this interpretation the twisted dual action is given by: 
\begin{equation}\label{intc}\hat\rho'^\Psi(x)=\Psi^*\hat\rho'(x)\Psi.\end{equation} The von Neumann algebra $\W^\infty(\mbX^\Psi)$ introduced above coincides with the algebra of the $\hat\rho'^\Psi$-coinvariants
\[\W^\infty(\mbX^\Psi)=\{x\in\Gamma\ltimes_{\rho'}\W^\infty(\mbX)|\,\hat\rho'^\Psi(x)=x\otimes 1\}.\]
Our next aim is to  extend  the $\C^*$-algebraic coaction $\Delta_{\mbD^\Psi}$ to the von Neumann  level and define $\Delta_{\W^\infty(\mbX^\Psi)}:\W^\infty(\mbX^\Psi)\rightarrow\W^\infty(\mbG^\Psi)\bar\tens\W^\infty(\mbX^\Psi)$. In order to do it we make the following observations:
\begin{itemize}
\item The crossed product $\Gamma\ltimes_{\rho'}\W^\infty(\mbX)$ may be faithfully represented on $\W^2(\mbG)$. Indeed, identifying  $\X$ with a quotient space $\G/\G_x$ where $\G_x$ is a stabilizing subgroup of a point $x\in X$ and using the idea of Section 4 of \cite{Vaes} one may check that there exists an injective normal $*$-homomorphism that sends $\W^\infty(\mbX)\subset\Gamma\ltimes_{\rho'}\W^\infty(\mbX)$ to $\W^\infty(\G/\G_x)\subset\B(\W^2(\mbG))$ and  $\lambda_\gamma\in\Gamma\ltimes_{\rho'}\W^\infty(\mbX)$ to the left shift $\W_\gamma\in\B(\W^2(\mbG)))$.
\item Let us note that the representation $\Gamma^2\ni(\gamma_1,\gamma_2)\mapsto \W_{\gamma_1}\otimes \W_{\gamma_2}\in\B(\W^2(\mbG)\otimes\W^2(\mbG))$ when extended to $\C^*(\Gamma^2)$ provides an interpretation of the unitary $\Upsilon\in\M(\C^*(\Gamma^2))$ (see \eqref{defup}) as an element of $\B(\W^2(\mbG)\otimes\W^2(\mbG))$. This in turn enables us to define  the unitary \[U=\Upsilon W^*(\Hat{J}J\otimes 1)V\in\B(\W^2(\mbG)\otimes\W^2(\mbG))\]
- for the definition of $J$ and $\Hat{J}$ we refer to Section \ref{prem}.  It turns out that $U$ implements  $\Delta_{\mbX^\Psi}$: $\Delta_{\mbX^\Psi}(x)=U(1\otimes x)U^*$. In order to prove  that  it is enough to check the following two equalities:
\begin{itemize}\item[(i)] $U(1\otimes\W_\gamma)U^*=\W_\gamma\otimes 1=\Delta_{\mbX^\Psi}(\lambda_\gamma)$,
\item[(ii)] $U(1\otimes f)U^*=\Upsilon\Delta_{\mbX}(f)\Upsilon^*=\Delta_{\mbX^\Psi}(f)$ for any $\gamma\in\Gamma$ and $f\in\C_0(\X)$.
\end{itemize} The sufficiency of that may be concluded from Section \ref{rdlg}.  
\end{itemize}
The above observations shows that we may define the  von Neumann counterpart of $\Delta_{\mbX^\Psi}$ by the formula 
$\Delta_{\W^\infty(\mbX^\Psi)}(x)=U(1\otimes x)U^*$ where $x\in\W^\infty(\mbX^\Psi)$. 
\begin{rem}
Using the embedding of $\Gamma\ltimes_{\rho'}\W^\infty(\mbX)$ into $\B(\W^2(\mbG))$ described above one may express the $\hat{\rho}'^\Psi$-invariance of $x\in\W^\infty(\mbX^\Psi)$ by the equation 
\begin{equation}\label{invx}W(x\otimes 1)W^*=\Psi(x\otimes 1)\Psi^*,\end{equation} 
where $W$ is the left regular corepresentation of $\mbG$.
It follows from  Eq. \eqref{intc} and the fact that $W$ implements the dual coaction $\hat\rho'(x)=W(x\otimes 1)W^*$.
\end{rem}
Let us now show that the coaction of $\Delta_{\W^\infty(\mbX^\Psi)}$ is ergodic. Suppose that $\Delta_{\W^\infty(\mbX^\Psi)}=1\otimes x$ for some $x\in\W^\infty(\mbX^\Psi)$. Note that $\Gamma$ is a  closed subgroup of $\mbG^\Psi$ in the sense of Definition 2.5 of \cite{Vaes}. In particular the $\mbG^\Psi$-invariance of  $x$ implies its $\Gamma$-invariance. In other words $x$ commutes with the image  of the representation $\lambda:\Gamma\rightarrow\Gamma\ltimes_{\rho'}\W^\infty(\mbX)$. This together with the  $\hat{\rho}'^\Psi$-invariance of $x$ shows that $x\in\W^\infty(\mbX)$ and $\Delta_{\W^\infty(\mbX)}(x)=1\otimes x$. The ergodicity of $\Delta_{\W^\infty(\mbX)}$ implies that $x\in\mathbb{C}1$ which proves the ergodicity of $\Delta_{\W^\infty(\mbX^\Psi)}$ . 

 The case of the Rieffel deformation of the quotient space $\G/\G_0$ for which $\Gamma\subset \G_0\subset \G$ was considered in \cite{pkrdhs}. It was proved there that under some modular assumption concerning the embedding $\Gamma\subset\G$  and some assumptions on $\Psi$ we have $(\G/\G_0)^\Psi\cong\mbG^\Psi/\mbG_0^\Psi$. By the results of \cite{Vaes}, the right hand side of this equality gives rise to a QHS. It turns out that the proof given in \cite{pkrdhs} may be simplified and adopted to the proof of the fact that $(\W^\infty(\mbX^\Psi),\Delta_{\W^\infty(\mbX^\Psi)})$ is a QHS where we do not assume that $\Gamma\subset\G_0$, there are no modular assumption concerning the embedding $\Gamma\subset\G$ and no assumption about a $2$-cocycle $\Psi$ except its continuity. In this more general situation we expect to obtain generically a QHS of the non-quotient type. Let us first prove the following lemma.
\begin{lem}\label{cpr} Let $\X$  be a $\G$-homogeneous space and $\mbX$ the corresponding $\mbG$-$\C^*$-algebra. Let $\Gamma\subset\G$ be a closed abelian subgroup of $\G$ and $\Psi$ a continuous $2$-cocycle on the dual group $\Hat\Gamma$. Let $\mbX^\Psi$ be the Rieffel deformation of $\mbX$. Then the corresponding crossed-product $\C^*$-algebras are isomorphic: $\mbG\ltimes\C_0(\mbX)\cong\mbG^\Psi\ltimes\C_0(\mbX^\Psi)$.
\end{lem}
\begin{proof}
Note that $\C_0(\widehat\mbG)$ and $\C_0(\mbX)$ may be embedded into $\M(\mbG\ltimes\C_0(\mbX))$ and they together generate it: $\mbG\ltimes\C_0(\mbX)=[\C_0(\widehat\mbG)\C_0(\mbX)]$. In order to  embed $\C_0(\mbX)$ we use the coaction $\Delta_{\mbX}$ - see \eqref{cacp}. The fact that $\Gamma\ltimes\C_0(\mbX)$ may also be embedded into $\M(\mbG\ltimes\C_0(\mbX))$, the equality   $\C_0(\widehat\mbG)=\C_0(\widehat\mbG^\Psi)$ together with the isomorphism \begin{equation}\label{cpe}\Gamma\ltimes\C_0(\mbX)\cong\Gamma\ltimes\C_0(\mbX^\Psi)\end{equation} (proved in  Proposition 3.2 of \cite{pkrd}) leads to
\begin{eqnarray*}\mbG\ltimes\C_0(\mbX)&=&[\C_0(\widehat\mbG)\C_0(\mbX)]=[\C_0(\widehat\mbG)\C^*(\Gamma)\C_0(\mbX)]=[\C_0(\widehat\mbG)(\Gamma\ltimes\C_0(\mbX))]\\&=&[\C_0(\widehat\mbG^\Psi)\C^*(\Gamma)\C_0(\mbX^\Psi)]=\mbG^\Psi\ltimes\C_0(\mbX^\Psi). \end{eqnarray*}  
Note that in the second and the fourth equality we used the fact that $\C_0(\widehat\mbG)=[\C^*(\Gamma)\C_0(\widehat\mbG)]$ and in the fourth equality we used $\C_0(\widehat\mbG^\Psi)=\C_0(\widehat\mbG)$.
\end{proof}
\begin{rem}\label{nt1} In this remark we give a concise description of the induction procedure of the regular representation $\G_0\ni g\mapsto\W^{\G_0}_g\in\mathcal{L}(\C_0(\widehat\mbG_0))$.
Since our ultimate aim  is to prove that $\mbX^\Psi$ gives rise to a quantum homogeneous  $\mbG^\Psi$-space, we shall stick to the S. Vaes' formulation and notation of the induction procedure for locally compact quantum groups given in \cite{Vaes}. We shall do so even in the case of the classical closed subgroup $\G_0$ of a locally compact group $\G$. 

The induction procedure applied to the left regular representation $\G_0\ni g\mapsto\W^{\G_0}_g\in\mathcal{L}(\C_0(\widehat\mbG_0))$  gives rise to the induced $\C_0(\widehat\mbG_0)$-module $\Ind(\C_0(\widehat\mbG_0))$ and the induced representation $\Ind(\W^{\G_0}):\G\rightarrow\mathcal{L}(\Ind(\C_0(\widehat\mbG_0)))$. Let $\alpha:\G\rightarrow\Aut(\W^\infty(\mbX))$ be the action given by the left shifts: $\alpha_g(f)(x)=f(g^{-1}x)$. As is always the case for the induced representations, $\Ind(\C_0(\widehat\mbG_0))$ is equipped with a faithful strict $*$-homomorphism $\rho:\W^\infty(\mbX)\rightarrow\mathcal{L}(\Ind(\C_0(\widehat\mbG_0)))$  which is covariant with respect to $\Ind(\W^{\G_0})$: $\rho(\alpha_g(f))=\Ind(\W^{\G_0})_{g}\rho(f)\Ind(\W^{\G_0})_{g}^*$. 

In the course of the proof of Theorem 6.1 of \cite{Vaes} it was noted that we have the following identification  \begin{equation}\label{cocr}\mathcal{K}(\Ind(\C_0(\widehat\mbG_0)))\cong\mbG\ltimes\C_0(\mbX).\end{equation}
In particular $\mathcal{K}(\Ind(\C_0(\widehat\mbG_0)))$ is equipped with the dual (right) coaction - the counterpart of the dual coaction on $\mbG\ltimes\C_0(\mbX)$: 
\begin{equation}\label{dualcoc}
\gamma:\mathcal{K}(\Ind(\C_0(\widehat\mbG_0)))\rightarrow\M(\mathcal{K}(\Ind(\C_0(\widehat\mbG_0)))\otimes\C_0(\widehat\mbG)).
\end{equation}

The data described above may be encoded in  the bicovariant $\C_0(\widehat\mbG_0)$-bicorrespondence $\mathcal{F}$ (see Section 3.1 of \cite{Vaes}). As a Hilbert $\C_0(\widehat\mbG_0)$-module, $\mathcal{F}$ is the tensor product $\mathcal{F}=\W^2(\mbG)\otimes\Ind(\C_0(\widehat\mbG_0))$. The structure of the bicovariant   $\C_0(\widehat\mbG_0)$-bicorrespondence on $\mathcal{F}$ is given by:
\begin{itemize}
\item a strict $*$-homomorphism $\pi_l:\W^\infty(\widehat\mbG)\rightarrow\mathcal{L}(\mathcal{F})$ which sends a generator $\W_g$ to $\pi_l(\W_g)=\W_g\otimes\Ind(\W^{\G_0})_g$;
\item a strict $*$-antihomomorphism $\pi_r:\W^\infty(\widehat\mbG)\rightarrow\mathcal{L}(\mathcal{F})$ which sends a generator $\W_g$ to $\pi_r(\W_g)=\R_g\otimes 1$;
\item a strict $*$-homomorphism $\pi:\W^\infty(\mbG)\rightarrow \mathcal{L}(\mathcal{F})$ given by $\pi(f)=f\otimes 1$ for any $f\in\W^\infty(\mbG)$.
\end{itemize}
The bicovariance relation linking $\pi_l,\pi_r$ and $\pi$ and the right regular corepresentation  $\Hat{V}\in\W^\infty(\mbG)\otimes\W^\infty(\widehat\mbG)$ (see Eq. \eqref{rrr}) is described in Definition 3.5 of \cite{Vaes}. 
 It may be shown that $\pi_l$ together with $\rho$ gives rise to a strict $*$-homomorphism $\sigma:\mbG\ltimes\W^\infty(\mbX)\rightarrow\mathcal{L}(\mathcal{F})$ such that $\sigma(\W_g)=\pi_l(\W_g)$ and $\sigma(f)=\rho(f)$ for any $g\in\G$ and $f\in\W^\infty(\mbX)$. In what follows $\sigma$ will be denoted by $\pi_l$. 
\end{rem}
\begin{thm} Let $\X$ be a $\G$-homogeneous space, $\Gamma\subset\G$  a closed abelian subgroup and $\Psi$ a continuous $2$-cocycle on the Pontryagin dual $\Hat\Gamma$. Let $\mbX^\Psi$ be the Rieffel deformation of\, $\mbX$ described above. Then the von Neumann version $(\W^\infty(\mbX^\Psi),\Delta_{\mbX^\Psi})$ of $\mbX^\Psi$ is a QHS with the corresponding $\C^*$-algebraic version coinciding with $\mbX^\Psi$. 
\end{thm}
\begin{proof}  In the course of the proof we shall use the notation introduced in Remark \ref{nt1}.
Let us consider the $\C_0(\widehat\mbG_0)$-module $\W^2(\mbG)\otimes\Ind(\C_0(\widehat\mbG_0))\otimes\W^2(\mbG)$. Note that under the identification \eqref{cocr} we have \begin{equation}\label{pil}\pi_l(x)_{12}=\Sigma_{13}\gamma_{23}(x)\Sigma_{13}\end{equation} for any $x\in\M(\mbG\ltimes\C_0(\mbX))$. On the other hand, using Eq. \eqref{invx} we may see that  \[\gamma(x)=\Psi(x\otimes 1)\Psi^*,\] for any  $x\in\C_0(\mbX^\Psi)$. This equation together with \eqref{pil} shows that, for any $x\in\C_0(\mbX^\Psi)$ we have $\Phi\pi_l(x)\Phi^*=1\otimes x$ where $\Phi\in\M(\C^*(\Gamma)\otimes\C^*(\Gamma))$  is the unitary element given by $\Phi(\hat\gamma_1,\hat\gamma_2)=\overline{\Psi(\hat\gamma_2,\hat\gamma_1)}$. Using the strictness of $\pi_l$ and Lemma \ref{cpr} we conclude that the natural embedding of $\C_0(\mbX^\Psi)$ into $\M(\mbG^\Psi\ltimes\C_0(\mbX^\Psi))$ extends to the embedding $\iota: \W^\infty(\mbX^\Psi)\rightarrow\M(\mbG^\Psi\ltimes\C_0(\mbX^\Psi))$. 

Let us note that $\iota$ inherits the following strictness property: for any uniformly bounded, $*$-strongly convergent net $x_i\in\W^\infty(\mbX^\Psi)$ and any $y\in\mbG^\Psi\ltimes\C_0(\mbX^\Psi)$, the net $\iota(x_i)y$ is norm convergent.  Composing $\iota$ with the embedding  \[\mbG^\Psi\ltimes\C_0(\mbX^\Psi)=[\Delta_{\mbX^\Psi}(\C_0(\mbX^\Psi))(\C_0(\widehat\mbG^\Psi)\otimes 1)]\hookrightarrow\M(\cml\otimes\C_0(\mbX^\Psi))\]  we conclude that the map $\Delta_{\W^\infty(\mbX^\Psi)}:\W^\infty(\mbX^\Psi)\rightarrow\W^\infty(\mbG^\Psi)\bar\tens\W^\infty(\mbX^\Psi)$ gives rise to a strict $*$-homomorphism $\Delta_{\W^\infty(\mbX^\Psi)}:\W^\infty(\mbX^\Psi)\rightarrow\mathcal{L}(\W^2(\mbG^\Psi)\otimes\C_0(\mbX^\Psi))$. This shows that third condition  of Definition \ref{qhs2} is satisfied. The first and the second are trivially satisfied in our case. 
\end{proof}
\end{subsection}
\end{section}

\end{document}